\documentclass[sn-mathphys-num]{sn-jnl}% Math and Physical Sciences Numbered Reference Style 
\usepackage{graphicx}%
\usepackage{multirow}%
\usepackage{amsmath,amssymb,amsfonts}%
\usepackage{amsthm}%
\usepackage{mathrsfs}%
\usepackage[title]{appendix}%
\usepackage{xcolor}%
\usepackage{textcomp}%
\usepackage{manyfoot}%
\usepackage{booktabs}%
\usepackage{algorithm}%
\usepackage{algorithmicx}%
\usepackage{algpseudocode}%
\usepackage{listings}%
\usepackage{bm}
\newtheoremstyle{mystyle}  % ????
{6pt}                   % ????
{6pt}                   % ????
{\normalfont}            % ????
{}                       % ??
{\bfseries}              % ????
{.}                      % ?????
{5pt plus 1pt minus 1pt} % ?????
{}                       % ????

\theoremstyle{mystyle}

\newtheorem{theorem}{Theorem}%  meant for continuous numbers
\newtheorem{lemma}{Lemma}%  meant for continuous numbers
\newtheorem{definition}{Definition}%

\begin{document}

\title[]{Exact distribution of discrete-time D-BMAP/G/$\infty$ queueing model}


\author[1]{\fnm{Tonglin} \sur{Liao}}%\email{iauthor@gmail.com}

\author*[1]{\fnm{Youming} \sur{Li}}\email{youmingli@uestc.edu.cn}

\affil*[1]{\orgdiv{School of Mathematical Sciences}, \orgname{University of Electronic Science and Technology of China}, \city{Chengdu}, \postcode{611731},  \country{China}}


\abstract{
	In this paper, we consider discrete-time D-BMAP/G/$\infty$ queueing model. We construct effective discrete-time Markovian dynamics for this model and utilize it to derive exact time-dependent distribution of customer number and the corresponding moments for the original queueing model. Numerical simulations
are used to verify our results. Using our result, we provide analytical distribution for discrete-time M/M/$\infty$, and then compare it with the distribution of continuous-time M/M/$\infty$.
}


\maketitle

\section{Introduction}\label{sec1}

Since Erlang's pioneering work \cite{erlang1909theory}, there have been considerable research efforts on continuous-time queueing models, which have laid the foundation for the analysis of stochastic service systems. Queueing theory has been extensively applied to diverse domains including information technologies, transportation systems, and modern management systems, where it continues to play a pivotal role in performance optimization and system design \cite{nazarov2014information,kerobyan2018infinite}.  Moreover, queueing theory has also been used in modeling and analyzing biological systems. Recent biological experiments have found non-Markovian phenomena in stochastic gene expression, and due to the great potential of queueing theory in solving non-Markovian models, many researchers studying stochastic gene expression models have begun to pay attention to this classical yet powerful field \cite{shi2023stochastic,shi2024nascent,szavits2024solving}.

To model complex arrival mechanisms in real world, general arrival processes are needed.
The batch Markovian arrival process (BMAP) introduced in \cite{lucantoni1991new} is capable of capturing the batch, correlated, and bursty characteristics of real-world arrival processes \cite{cao2018joint}, and it covers a wide range of well-known arrival processes, including Poisson process (M), Phase-type renewal process (PH), Markov-modulated Poisson process (MMPP), and Markovian arrival process (MAP) \cite{kerobyan2018infinite,schwartz1986telecommunication,neuts1979versatile}. 

While continuous-time queueing models have been extensively studied in the literature, discrete-time queueing systems, whose emergence can be traced back to the year 1958 \cite{meisling1958discrete}, have also attracted increasing attention for their applicability in digital communication and computer systems. In communication network, digital information is often disseminated in fixed-length ``packets'' \cite{harini2023literature}, requiring a fixed-length transmission time known as slots. For example, when studying signal transmission and reception, the transmission time for a packet has been observed to be close to 5.5 ms \cite{wu2016end}, which highlights the discrete nature of such transmission processes. Due to its suitability for modeling computer systems and performance analysis, discrete-time queues have attracted considerable attention from queueing theorists and communication engineers.

Most literature on discrete queueing theory focus on the analysis of single-server systems based on matrix-analytic
methods \cite{ saffer2010unified,hunter2014mathematical,takagi1994analysis,harini2023literature}. For example, there have been recent articles studying the application of MMBP/G/1 in communication networks \cite{wu2016end,wang2011discrete}, where MMBP refers to the Markov-modulated Bernoulli process. On the other hand, multi-server and infinite-server queueing models are equally important \cite{schwartz1986telecommunication} due to their practical applications in various fields. For instance, discrete-time infinite-server queueing models have been used to analyze healthcare demand \cite{worthington2020infinite}.

In this paper, we study the exact solution of the discrete-time D-BMAP/G/$\infty$ queueing model. The main idea is to construct the effective Markovian dynamics of the queueing model, and then solve it by classical methods to obtain the solution for the original queueing model.
%where D refers to discrete, using the effective Markovian dynamics (EMD) method introduced in \cite{li2024effective}. This method   is similar to the method of thinning process used  in \cite{kerobyan2018infinite}, but the EMD method explicitly clarifies the difference between the effective and original processes, and highlights the Markovian nature of the effective process. 
The paper is organized as follows: In Section 2 we introduce the discrete-time queueing model under consideration. In Section 3 we establish the effective Markovian dynamics for the queueing models. In Section 4 we provide the analytical results for D-BMAP/G/$\infty$. In Section 5 we apply our results to some examples.

\section{Discrete-time D-BMAP/G/$\infty$}\label{sec:model}

In this paper we consider discrete-time queueing models with infinite servers. Different from classical continuous-time queueing models, for discrete-time queueing models, time axis is segmented into uniform intervals called time slots. The time axis is denoted by $0,1,..., t, ...$ with $t$ being integer, and the inter-arrival and service times are both non-negative integer-valued random variables. 

We are interested in the distribution of the number of customers present in the system at specific time slots. Since there are infinitely many servers, each arriving customer begins service immediately upon arrival and is assigned an independent service time.

\begin{definition}
Let \( Y \) be the service time of a customer taking values in \{1, 2, \dots\} with probability mass function $
\mathbb{P}\{Y = k\}$. The survival function of the service time distribution is defined by $\Phi(t)=\mathbb{P}(Y>t)$.
\end{definition}

Consider a customer arriving at time slot~\( t \) with service time \( s \). Such a customer is counted at time slots \( t, t+1, \dots, t+s-1 \). but starting from time slot \( t+s \), the customer will no longer be counted, as its service is completed.

The arrival process is governed by a discrete-time batch Markovian arrival process (D-BMAP), a versatile and mathematically tractable model for capturing correlated and bursty arrivals in discrete-time queueing systems \cite{cao2018joint,saffer2010unified}. The D-BMAP framework is built upon a background Markov chain whose state transitions are synchronized with possible batch customer arrivals (i.e., multiple customer arrivals per time slot). Let the state space of the background Markov chain be $\{1,2,\dots\}$, and we let $d_{ij}(l)$ be the probability that there are $l$ customer arrivals along with the transition from $i$ to $j$. Then the synchronized batch customer arrivals can be fully described by a set of arrival probability matrices $\{\mathbf{D}_l\}_{l=0}^{\infty}$, where $\mathbf{D}_0$ is the matrix with elements $d_{ij}(0)$ governs transitions corresponding to no arrivals, and  $\mathbf{D}_l$ with $l\geq 1$ is the matrix with elements $d_{ij}(l)$ that governs transitions corresponding to arrivals of batches of size $l$. Clearly, any D-BMAP is completely determined by the matrix sequence $\{\mathbf{D}_l\}_{l=0}^\infty$. \;\\

\begin{definition}
	The matrix generating function of the $\{\mathbf{D}_l\}_{l=0}^\infty$ is defined by
	\begin{equation}\label{dz}
	D(z)=\sum_{l=0}^\infty \mathbf{D}_lz^l,
	\end{equation}
	with its $(i,j)$-element being $D_{ij}(z)$ defined by
	\begin{equation}\label{dij}
	D_{ij}(z)=\sum_{l=0}^\infty d_{ij}(l)z^l.
	\end{equation}
\end{definition}
Clearly, there exists a bijective correspondence between the matrix sequence $\{\mathbf{D}_l\}_{l=0}^\infty$ and its generating function $D(z)$, and $\mathbf{D}_l$ characterizes the batch size probabilities $d_{ij}(l)$ associated with the transition from $i$ to $j$. Moreover, the matrix $D(1)=\sum_{l=0}^\infty \mathbf{D}_l$ is actually the transition probability matrix for the background Markov chain governing the queueing state, and for this reason we also write the matrix as $P$. 

We stress here that for discrete-time queueing models, when the inter-arrival time or service time distribution is described by ${\rm M}$ using Kendall's notation, then the distribution is actually geometrically distributed since the geometric distribution is the unique memoryless discrete distribution. Specifically, the memoryless service process means that any arrived customer completes its service in each time slot with a fixed probability, regardless of the history, making the total time of leaving follows a geometric distribution.

The D-BMAP actually covers many arrival processes. For example, when a D-BMAP is constrained to allow only $0$ or $1$ arrivals per slot, it reduces to an Markov-Modulated Bernoulli Process (MMBP) with \cite{wu2016end}
\begin{equation*}
D(z)=\mathbf{D}_0+\mathbf{D}_1z.
\end{equation*}

Since in this paper the service times of customers can follow an arbitrary distribution, we assume that there is no customer at the initial time; otherwise, we would need additional information to determine the completion times of these services.

\section{Effective Markovian dynamics for the non-Markovian queueing model}

To describe the state in the queueing model, we need the following definitions.

\begin{definition}
	Let $N(t)$ denote the number of customers in the system at time $t$, and let $I(t)$ be the state of the background Markov chain at time $t$.
\end{definition}

In this paper we are interested in $p_{m}(t)=\mathbb{P}\bigl(N(t) = m\bigr)$, the probability of having $m$ customers at time $t$. Note that $N(t)$ and $I(t)$ are coupled, we actually need to discuss $p_{i,m}(t)=\mathbb{P}\bigl(I(t)=i,N(t) = m\bigr)$. Clearly we have $p_m(t)=\sum_{i=1}^\infty p_{i,m}(t)$.
Since the service times in the queueing models can follow general distribution, the binary process $(I(t),N(t))$ is generally non-Markovian and thus challenging to analyze analytically. We now establish effective Markovian dynamics for the binary process $(I(t),N(t))$ to analytically derive $p_{m}(t)$ \cite{li2024effective}.
\begin{definition}
Let \(N(s;t)\) denote the number of customers present in the system at time \(s\) who will remain in service at time \(t\).
\end{definition}
In the rest of the paper we call \( N(s;t) \) the effective process. The following lemma characterizes the effective process $N(s;t)$.
\begin{lemma}\label{le1}
	The effective process $N(s;t)$ has the following three basic properties:
\begin{enumerate}
		\item  For each fixed time slot $t > 0$, the effective process is defined with $0\leq s\leq t$.
		\item For any \( s \leq t \), we have \( N(s;t) \leq N(s) \). In particular, when $s=t$ we have
		\[
		N(t;t) = N(t).
		\]
	\item  For any \(0 \leq s_1 \leq s_2 \leq t\), we have \(N(s_1;t) \leq N(s_2;t)\).
	%since all customers counted in $N(s_1;t)$ will remain in service at time $t$,  and additional customers arriving during $(s_1,s_2]$ may further increase the count.
\end{enumerate}
\end{lemma}

\begin{proof}
	The first statement in Lemma \ref{le1} is obviously true. We then prove the second statement. Let $A(s)$ be counting process representing the total number of customer arrivals up to time $s$, and let $\tau_i$ and $Y_i$ be the arrival time and service time for $i$-th customer. Then the original process $N(s)$ can be written as
	\[
	N(s) = \sum_{i=1}^{A(s)} \mathbb{I}_{\{\tau_i \leq s < \tau_i + Y_i\}},
	\]
	where $\mathbb{I}$ is the indicator function. With these notation, the effective process can also be written by
	\[
	N(s;t) = \sum_{i=1}^{A(s)} \mathbb{I}_{\{\tau_i \leq s < \tau_i + Y_i\}} \cdot \mathbb{I}_{\{Y_i > t - \tau_i\}}.
	\]
	Then we have
	\begin{equation*}
	\begin{aligned}
	N(s;t) &= \sum_{i=1}^{A(s)} \mathbb{I}_{\{\tau_i \leq s < \tau_i + Y_i\}} \cdot \mathbb{I}_{\{Y_i > t - \tau_i\}}\leq \sum_{i=1}^{A(s)} \mathbb{I}_{\{\tau_i \leq s < \tau_i + Y_i\}} = N(s).
	\end{aligned}
	\end{equation*}
	Moreover, taking $s=t$ yields
	\begin{equation*}
	N(t;t)=\sum_{i=1}^{A(t)} \mathbb{I}_{\{\tau_i \leq t < \tau_i + Y_i\}} \cdot \mathbb{I}_{\{Y_i > t-\tau_i\}}=\sum_{i=1}^{A(t)} \mathbb{I}_{\{\tau_i \leq t < \tau_i + Y_i\}}=N(t),
	\end{equation*}
since $\tau_i\leq t<\tau_i+Y_i$ always yields $Y_i>t-\tau_i$, and this completes the proof of the second statement. 
	To prove the third statement, we consider \( 0 \leq s_1 \leq s_2 \leq t \). Recall that we can write $N(s_1;t)$ and $N(s_2;t)$ as
	\[
	N(s_1;t) = \sum_{i=1}^{A(s_1)} \mathbb{I}_{\{\tau_i \leq s_1 < \tau_i + Y_i\}} \cdot \mathbb{I}_{\{Y_i > t - \tau_i\}},
	\]
	\[
	N(s_2;t) = \sum_{i=1}^{A(s_2)} \mathbb{I}_{\{\tau_i \leq s_2 < \tau_i + Y_i\}} \cdot \mathbb{I}_{\{Y_i > t - \tau_i\}}.
	\]
Clearly, we have \( A(s_1) \leq A(s_2) \), and this gives
	\begin{equation}\label{eq4}
	N(s_2;t) = \sum_{i=1}^{A(s_1)} \mathbb{I}_{\{\tau_i \leq s_2 < \tau_i + Y_i\}} \cdot \mathbb{I}_{\{Y_i > t - \tau_i\}} + \sum_{i=A(s_1)+1}^{A(s_2)} \mathbb{I}_{\{\tau_i \leq s_2 < \tau_i + Y_i\}} \cdot \mathbb{I}_{\{Y_i > t - \tau_i\}}.
	\end{equation}
	For each \( i \leq A(s_1) \) with $\tau_i \leq s_1 < \tau_i + Y_i$ and $ Y_i > t - \tau_i$, we have $\tau_i\leq s_1\leq s_2$ and $Y_i+\tau_i>t\geq s_2$, then we have that
	\begin{equation}\label{eq5}
	\mathbb{I}_{\{\tau_i \leq s_1 < \tau_i + Y_i\}} \cdot \mathbb{I}_{\{Y_i > t - \tau_i\}}\leq \mathbb{I}_{\{\tau_i \leq s_2 < \tau_i + Y_i\}} \cdot \mathbb{I}_{\{Y_i > t - \tau_i\}}.
\end{equation}
Combining Eqs. \eqref{eq4} and \eqref{eq5} we obtain that	
\begin{equation*}
\begin{aligned}
	\;&N(s_1;t)= \sum_{i=1}^{A(s_1)} \mathbb{I}_{\{\tau_i \leq s_1 < \tau_i + Y_i\}} \cdot \mathbb{I}_{\{Y_i > t - \tau_i\}}\\
	\;& \leq \sum_{i=1}^{A(s_1)} \mathbb{I}_{\{\tau_i \leq s_2 < \tau_i + Y_i\}} \cdot \mathbb{I}_{\{Y_i > t - \tau_i\}} + \sum_{i=A(s_1)+1}^{A(s_2)} \mathbb{I}_{\{\tau_i \leq s_2 < \tau_i + Y_i\}} \cdot \mathbb{I}_{\{Y_i > t - \tau_i\}}=N(s_2;t).
\end{aligned}
\end{equation*}
	This completes the proof.
\end{proof}

We now intuitively explain Lemma \ref{le1}. The second statement $N(s;t)\leq N(s)$ holds since those customers at time $s$ leaving before or at time $t$ will not be counted in $N(s;t)$. In addition,
 $N(s_1;t)\leq N(s_2;t)$ holds since all customers counted in $N(s_1;t)$ will remain in service at time $t$, therefore all of them will also all be counted in $N(s_2;t)$ and additional customers arriving during $(s_1,s_2]$ may further increase the count. In Fig. \ref{fig1} we present the comparison of trajectories of the original and effective processes for illustration.

\begin{figure}[htb!]
	\centering\includegraphics[width = 0.8\textwidth]{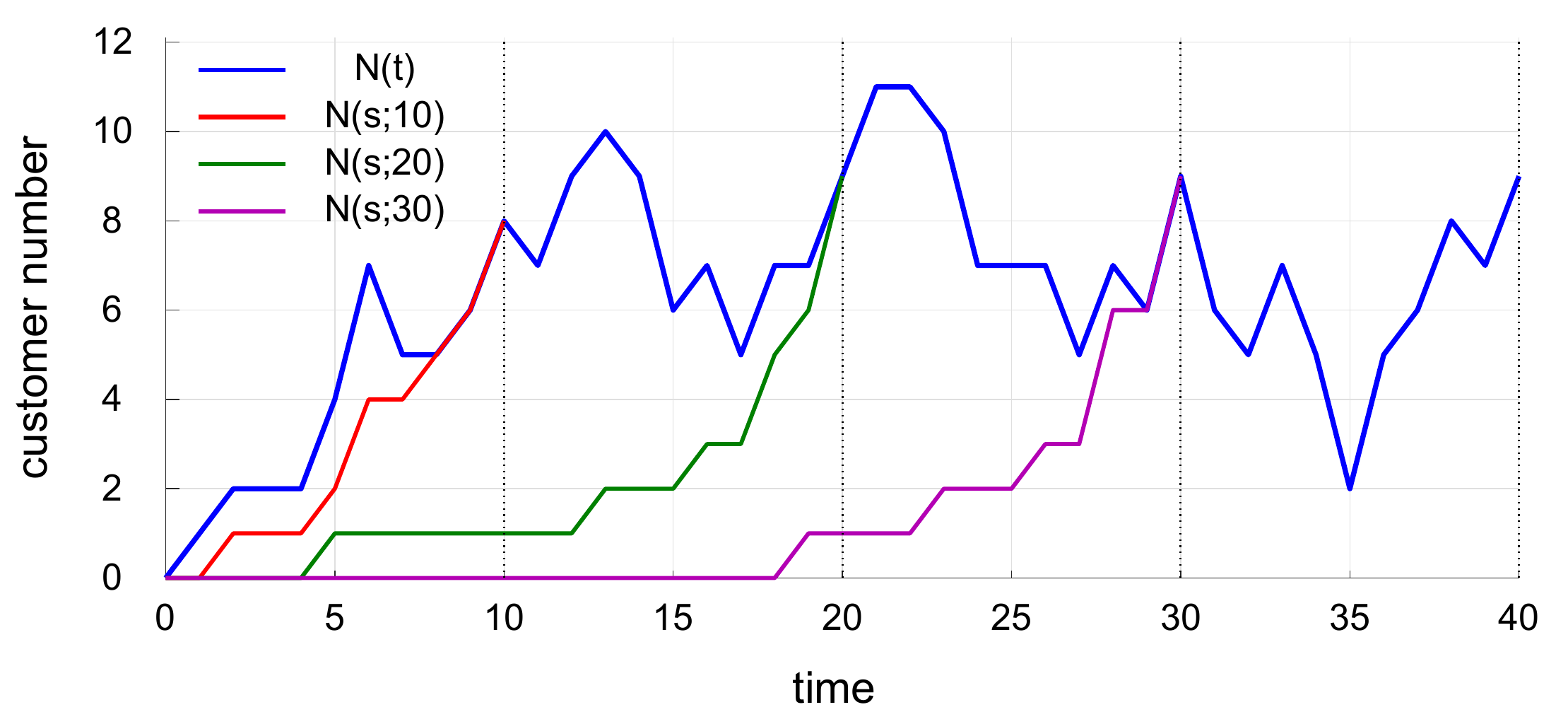}
	\caption{\textbf{Illustration of the effective process $N(s;t)$ for D-BMAP/G/$\infty$.} Here the blue curve is a trajectory of $N(t)$, and the other three curves are the trajectories of $N(s;t)$ with different $t$. Here the D-BMAP is given by $D_0 = [0.1, 0.2; 0.05, 0.2],  		D_1 = [0.1, 0.1; 0.1, 0.1],
		D_2 = [0.1, 0.15; 0.1, 0.1],  
		D_3 = [0.05, 0.1; 0.1, 0.05], 
		D_4 = [0.05, 0.05; 0.1, 0.1]$, and the service times are Poissonian distributed with mean value being $4$.}\label{fig1}
\end{figure}

It thus follows from the property $N(t;t)=N(t)$ that if we can obtain the distribution of $N(s;t)$, then by taking $s=t$ we will obtain the distribution of $N(t)$ for the original queueing model. The main reason of introducing the effective process $N(s;t)$ can be seen from the following lemma.
\begin{lemma}\label{le2}
	The effective process $(I(s),N(s;t))$ is Markovian, namely, for any state $(i,m)$ and $(j,n)$ with $n\geq m$, there exists a probability $p_{(i,m),(j,n)}$ such that 
	\begin{equation}\label{eq1}
	\begin{split}
	\mathbb{P}\Bigl(I(s+1) = j,\, N(s+1;t) = n \mid I(s)=i,\, N(s;t)=m,\, \mathcal{F}_{s}\Bigr)=p_{(i,m),(j,n)},
	\end{split}
	\end{equation}
	where $\mathcal{F}_s = \sigma\{(I(u),N(u;t)):u\leq s\}$ is a $\sigma$-algebra which contains all information of the effective process up to time $s$.
	\end{lemma}
\begin{proof}
	Note that the transition from state $(i, m)$ to state $(j, n)$ with $n\geq m$ occurs if and only if when the system transitions from $i$ and $j$ with exactly $n-m$ customers who remain in service at time $t$ arriving. Clearly, the latter event can be achieved by having $l$ customers arriving with $l\geq n-m$ and meanwhile only $n-m$ customers among them remain in service at time $t$. 	Since the state transition, customer arrival, and the service times are mutually independent, we have that 
	\begin{equation}\label{eq2}
\begin{aligned}
	&\mathbb{P}\Bigl(I(s+1) = j,\, N(s+1;t) = n \mid I(s)=i,\, N(s;t)=m,\, \mathcal{F}_{s}\Bigr) \\
	 =\;& \sum_{l=n-m}^\infty \mathbb{P}\Bigl(I(s+1) = j,\, N(s+1)-N(s) = l \mid I(s)=i\Bigr)\\
	&\times \mathbb{P}\Bigl(N(s+1;t)-N(s;t)=n-m\mid N(s+1)-N(s) = l\Bigr)\\
	=\;&	\sum_{l=n-m}^\infty d_{ij}(l)\binom{l}{n-m}\Phi(t-s)^{\,n-m}\Bigl[1-\Phi(t-s)\Bigr]^{\,l-(n-m)}\triangleq p_{(i,m),(j,n)},
\end{aligned}
	\end{equation}
	where we have considered all cases in which exactly \(n-m\) customers remain in service at time $t$ from a batch of \(l\) arrivals with $l\geq n-m$ and we have used the fact that for a customer arriving at time $s$ to remain in service at time $t$, then its service time much be strictly larger than $t-s$. Combining Eqs. \eqref{eq1} and \eqref{eq2} we complete the proof.
\end{proof}

Lemma \ref{le2} shows that the evolution of the effective binary process $(I(s),N(s;t))$ is fully determined by its present state, therefore it is indeed Markovian. Using Lemma \ref{le2} we can establish the following update equation:
	\begin{equation}
	\begin{aligned}
	&\;p_{j,n}(s+1;t)=\sum_{i=1}^\infty \sum_{m=0}^n p_{i,m}(s;t)p_{(i,m),(j,n)}\\ &\;=\sum_{i=1}^{\infty} \sum_{m=0}^{n} p_{i,m}(s;t) \sum_{l=n-m}^{\infty} d_{ij}(l) \binom{l}{n-m}\Phi(t-s)^{\,n-m}\Bigl[1-\Phi(t-s)\Bigr]^{\,l-(n-m)}.
	\label{eq:markov_update}
	\end{aligned}
	\end{equation}
The above update equation shows that $p_{j,n}(s+1;t)$ can be computed recursively, but the update equation itself seems to be very complicated. We then provide simplified expression of the update equation by using probability generating functions (PGF).

\section{Main Results}\label{sec:pgf}
In this section we show how to compute the exact probability distribution of the original queueing model.
% we iterate the evolution of the joint state \((I(s), N(s;t))\) from the initial time \(s=0\) (with an appropriate initial condition) up to \(s=t\). At this terminal point, the state variable \(N(t;t)\) exactly recovers the system state \(N(t)\), thereby yielding the transient distribution of the original system. This demonstrates that the EMD system provides a complete and exact Markovian representation of the non-Markovian underlying dynamics.
\begin{definition} For each state \(j\) we define the state-dependent generating functions for $N(s;t)$ and $N(t)$ as
	\[
	G_j(z,s;t) = \sum_{n=0}^{\infty} p_{j,n}(s;t) z^n,\;\;\;G_j(z,t) = \sum_{n=0}^{\infty} p_{j,n}(t) z^n,
	\]
	respectively.
\end{definition}
Then the generating functions for $N(s;t)$ and $N(t)$ can be given by
\begin{equation*}
G(z,s;t)=\sum_{j=1}^\infty G_j(z,s;t),\;\;\;G(z,t)=\sum_{j=1}^\infty G_j(z,t).
\end{equation*}
It then follows from Lemma \ref{le1} that $G(z,t)=G(z,t;t)$, hence we only need to obtain $G(z,s;t)$ with $0\leq s\leq t$.

The following theorem provides the update equation with respect to state-dependent generating functions $G_j(z,s;t)$.
\begin{theorem}\label{thm1}
	The state-dependent generating functions \(G_j(z,s;t)\) satisfy the following recursive relation:
	\begin{equation}\label{main}
	G_j(z,s+1;t) = \sum_{i=1}^{\infty} G_i(z,s;t)\, D_{ij}\bigl(\Phi(t-s)\,z + 1 - \Phi(t-s)\bigr),
	\end{equation}
	where $D_{ij}(z)$ is defined in Eq. \eqref{dij}.
\end{theorem}

\begin{proof}
	Multiplying $z^n$ on both sides of Eq. \eqref{eq:markov_update} and then summing over $n$ gives
	\begin{equation*}
	\begin{aligned}
	&\;	G_j(z, s+1; t) \\
	&\;= \sum_{n=0}^{\infty} z^n \sum_{i=1}^{\infty} \sum_{m=0}^n p_{i,m}(s;t)
		\sum_{l=n-m}^{\infty} d_{ij}(l) \binom{l}{n-m} \Phi(t-s)^{n-m} [1 - \Phi(t-s)]^{l-n+m}.\\
	\end{aligned}
	\end{equation*}
	By interchanging the summation order and applying the substitution $k = n - m$, we can rewrite the above equation as
	\[
	G_j(z,s+1;t)=\sum_{i=1}^{\infty} \sum_{m=0}^{\infty}  p_{i,m}(s;t)z^m
	\sum_{k=0}^{\infty} z^k \sum_{l=k}^{\infty} d_{ij}(l) \binom{l}{k} \Phi(t-s)^k\left[1-\Phi(t-s)\right]^{l-k},
	\]
	where we have divided $z^n$ as $z^m\times z^k$. Rearranging the sums in the above equations and applying the binomial theorem, we obtain that
	\begin{equation*}
	\begin{aligned}
	G_j(z,s+1;t)&\;=\sum_{i=1}^{\infty} \sum_{m=0}^{\infty}  p_{i,m}(s;t)z^m
	\sum_{l=0}^{\infty} d_{ij}(l) \sum_{k=0}^{l}  \binom{l}{k} \Phi(t-s)^k\left[1-\Phi(t-s)\right]^{l-k}z^k\\
	&\;=\sum_{i=1}^{\infty} G_i(z,s;t)
	\sum_{l=0}^{\infty} d_{ij}(l) \Bigl[\Phi(t-s)z+1-\Phi(t-s)\Bigr]^l.\\
\end{aligned}
\end{equation*}
The proof is then completed by noting that the inner sum is actually the generating function \(D_{ij}(z)\) evalutated at \(\Phi(t-s)z+1-\Phi(t-s)\).
\end{proof}

We then use Theorem \ref{thm1} to analytically derive $G(z,t)$ for the original non-Markovian queueing models. 
\begin{definition} 
	Let
\begin{equation*}
\begin{aligned}
\mathbf{g}(z,s;t)&\;=[G_1(z,s;t),G_2(z,s;t),\dots,G_j(z,s;t),\dots],\\
\mathbf{g}(z,t)&\;=[G_1(z,t),G_2(z,t),\dots,G_j(z,t),\dots],\\
\end{aligned}
\end{equation*}
be the vectors of state-dependent generating functions for the effective and original processes, respectively.
\end{definition}
According to the definitions above, the generating function $G(z,t)$ is given by $G(z,t)=\mathbf{g}(z,t)\bm{1}^T$, where $\bm{1}$ is the row vector with all its elements being $1$. Recall that we have assumed that there is no customer at time $0$, therefore $\mathbf{g}(z,0;t)=\mathbf{g}(z,0)$ and $\mathbf{g}(z,0)=(p_1(0),p_2(0),\dots)$ is simply a constant vector with its $j$-th element $p_j(0)$ being the initial probability for the queueing model to stay at the state $j$. For this reason we write $\mathbf{g}(z,0)=\mathbf{p}_0$.

It then follows from Eq. \eqref{main} that
\begin{equation}\label{eq3}
\mathbf{g}(z,s+1;t)=\mathbf{g}(z,s;t)\, T(z,s;t),
\end{equation}
where $T(z,s;t)$ is a matrix defined by
\begin{equation}\label{t}
T(z,s;t) = D\Bigl(\Phi(t-s)z+1-\Phi(t-s)\Bigr),
\end{equation}
with $D(z)$ being defined in Eq. \eqref{dz}. By iteratively using Eq. \eqref{eq3} and then applying Lemma \ref{le1} we finally obtain the following main result of this paper.

\begin{theorem}\label{thm2}
	The vector of state-dependent generating functions $\mathbf{g}(z,t)$ for the original discrete-time queueing model can be explicitly written as
	\begin{equation}\label{main2}
\mathbf{g}(z,t)= \mathbf{p}_0\prod_{k=0}^{t-1}T(z,k;t).
	\end{equation}
\end{theorem}
The time-dependent distribution $p_m(t)$ can then be recovered from the generating function $G(z,t)$ by taking derivatives as follows:
\begin{equation}\label{dist}
p_m(t)=\frac{1}{m!} G^{(m)}(0,t).
\end{equation}
Note that $t$ is a discrete variable, hence the derivative can only be taken with respect to $z$. Taking $t\rightarrow \infty$ in $p_m(t)$, we obtain the stationary distribution of customer number. 

We then discuss the moments of $p_m(t)$. To proceed, we let
$\mu_k(t)$ be the $k$-th factorial moment of the customer number distribution defined by
\begin{equation*}
\mu_k(t)=\sum_{m=0}^\infty m(m-1)\dots (m-k+1)p_m(t).
\end{equation*}
The reason why we use factorial moments here is that they can be directly recovered from the generating function as follows:
\begin{equation}\label{moment}
\mu_k(t)=G^{(k)}(1,t).
\end{equation}
To give the exact expressions of the distribution and the moments, we need to discuss the high-order derivatives of the generating function. Recall that $\mathbf{g}(z,0)$ is a constant vector $\mathbf{p}_0$, then by applying generalized Leibniz's rule to Eq. \eqref{main2} we obtain
\begin{equation*}
 \mathbf{g}^{(m)}(z,t)= 
\mathbf{p}_0
\sum_{\substack{l_0 + l_1 + \cdots + l_{t-1} = m \\ l_k \geq 0}} 
\frac{m!}{l_0! l_1! \cdots l_{t-1}!} 
\prod_{i=0}^{t-1}  T^{(l_i)}(z,i;t).
\end{equation*}
We stress here that the product order must be strictly maintained (multiplied from $0$ to $t-1$ in sequence), since the matrix multiplication is non-commutative. Specifically, for a given sequence $l_0,l_1,\dots,l_{t-1}$, the product is given by
\begin{equation*}
\prod_{i=0}^{t-1}  T^{(l_i)}(z,i;t)=T^{(l_0)}(z,0;t)T^{(l_1)}(z,1;t)\dots T^{(l_{t-1})}(z,t-1;t).
\end{equation*}
We obtain from Eq. \eqref{t} that
\begin{equation*}
T^{(l_i)}(z,i;t)=\Phi(t-i)^{l_i} D^{(l_i)}\left(\Phi(t-i)z+1-\Phi(t-i)\right).
\end{equation*}
Combining the above equations we arrive at
\begin{equation}\label{deri}
\begin{aligned}
\mathbf{g}^{(m)}(0,t)&\;= 
\mathbf{p}_0
\sum_{\substack{l_0 + l_1 + \cdots + l_{t-1} = m \\ l_i \geq 0}} 
\frac{m!}{l_0! l_1! \cdots l_{t-1}!} 
\prod_{i=0}^{t-1} \Phi(t-i)^{l_i} D^{(l_i)}\left(1-\Phi(t-i)\right),\\
\mathbf{g}^{(m)}(1,t)&\;= 
\mathbf{p}_0
\sum_{\substack{l_0 + l_1 + \cdots + l_{t-1} = m \\ l_i \geq 0}} 
\frac{m!}{l_0! l_1! \cdots l_{t-1}!} 
\prod_{i=0}^{t-1} \Phi(t-i)^{l_i} D^{(l_i)}(1).
\end{aligned}
\end{equation}
We stress here that $D^{(l_i)}(1)$ are exactly the matrix of factorial moments for state-transition for the background Markov chain. Inserting Eq. \eqref{deri} into Eqs. \eqref{dist} and \eqref{moment} gives the exact distribution and the corresponding moments, respectively. We now provide exact mean customer number and the variance. Taking $m=1,2$ in Eq. \eqref{deri} we obtain
\begin{equation*}
\begin{aligned}
\mathbf{g}^{(1)}(1,t)=&\; 
\mathbf{p}_0
\sum_{i=0}^{t-1}\Phi(t-i)P^iD^{(1)}(1)P^{t-1-i},\\
\mathbf{g}^{(2)}(1,t)=&\; 
\mathbf{p}_0\left[
2\sum_{0\leq i<j\leq t-1}\Phi(t-i)\Phi(t-j)P^{i-1}D^{(1)}(1)P^{j-i-1}D^{(1)}(1)P^{t-1-j}\right.\\
&\;\left.+\sum_{i=0}^{t-1}\Phi(t-i)^2P^iD^{(2)}(1)P^{t-1-i}\right].
\end{aligned}
\end{equation*}
Finally, using the above expressions, the mean customer number $m(t)$ and the variance $\sigma^2(t)$ can be explicitly written as
\begin{equation*}
\begin{aligned}
m(t)&\;=\mathbf{g}^{(1)}(1,t)\bm{1}^T,\\
\sigma^2(t)&\;=\left[\mathbf{g}^{(2)}(1,t)+\mathbf{g}^{(1)}(1,t)-\left(\mathbf{g}^{(1)}(1,t)\right)^2\right]\bm{1}^T.
\end{aligned}
\end{equation*}

\section{Examples}

In this section we apply our results to some examples. First we consider a two-state D-BMAP queueing model with the following generating function:
\begin{equation}\label{example1}
D(z) = \begin{bmatrix}
0.6\times (0.7 + 0.3z)^{10} & 0.4 \times(0.7 + 0.3z)^{10} \\
0.1\times (0.4 + 0.6z)^{20} & 0.9 \times(0.4 + 0.6z)^{20}
\end{bmatrix}.
\end{equation}
This above matrix completely determines the transition dynamics of the D-BMAP queueing model. For example, the first row of $D(z)$ shows that at each time slot, the queueing model transitions from state 1 to states 1 and 2 with probabilities 0.6 and 0.4, respectively, and the batch arrival follows a Binomial(10, 0.3). Let $Y$ be the service time, we assume that $Y-1$ is Poissonian distributed with parameter being $2$ to guarantee that $Y$ takes value in $\{1,2,\dots\}$. In this case, the survival function of service time is given by
\begin{equation}\label{ex1}
\Phi(t)=\sum_{i=t}^\infty \frac{2^i}{i!}e^{-2}.
\end{equation}
In Fig. \ref{fig2} we present the comparison between the exact results obtained by our results and the numerical results obtained by stochastic simulation algorithm (SSA) to validate our results. Clearly, our results agree perfectly with the numerical ones.

\begin{figure}[htb!]
	\centering\includegraphics[width = 1\textwidth]{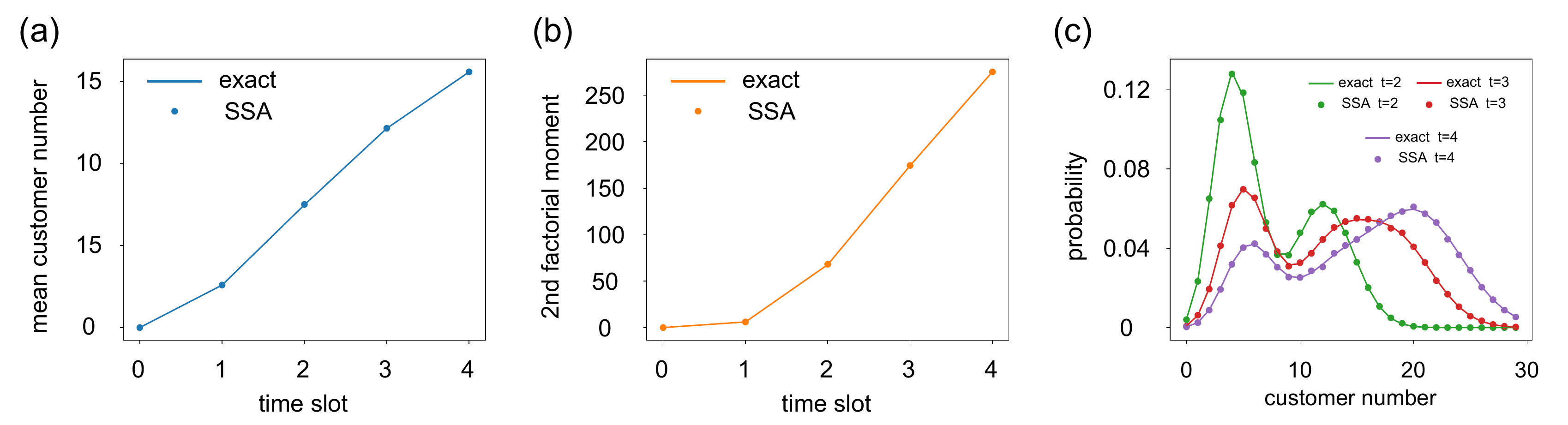}
	\caption{\textbf{Comparison between exact and numerical results for a two-state D-BMAP/G/$\infty$ queueing model given by Eqs. \eqref{example1} and \eqref{ex1}.} Here (a) compares the mean customer numbers, (b) compares the second-order factorial moments, and (c) compares the exact time-dependent distributions. The numerical results is obtained by averaging over 50000 realizations. The details of the SSA for D-BMAP/G/$\infty$ can be found in Appendix \ref{ssa}.}\label{fig2}
\end{figure}

It is well-known that the stationary distribution for the continuous-time ${\rm M}/{\rm M}/\infty$ queueing model is Poissonian distributed. We now consider the discrete-time ${\rm M}/{\rm M}/\infty$ queueing model and then make comparison. Recall that when inter-arrival time or service time distribution is described by ${\rm M}$ using Kendall's notation, then the distribution is actually geometrically distributed. We assume that the parameters of geometric distributions for the arrival and service processes are given by $p$ and $\alpha$, respectively. Then in this case we have
\begin{equation*}
D(z)=1-p+pz,\;\;\;\Phi(t)=\alpha^t.
\end{equation*}
It follows from Theorem \ref{thm2} that
\begin{equation}\label{mm}
\begin{aligned}
G(z,t)=\prod_{i=1}^{t}\left[1+p\alpha^{i}(z-1)\right],\\
\end{aligned}
\end{equation}
which means that the time-dependent distribution is actually the distribution of a sum of independent but not identically distributed Bernoulli trials.  By taking derivatives we have that
\begin{equation*}
\begin{aligned}
m(t)&\; = p \alpha \frac{1 - \alpha^t}{1 - \alpha},\\ 
\sigma^2(t)&\; = p\alpha \frac{1 - \alpha^t}{1 - \alpha} - p^2\alpha^2 \frac{1 - \alpha^{2t}}{1 - \alpha^2}.
\end{aligned}
\end{equation*}
The Fano factor is defined as the ratio of the variance to the mean of a random variable, measuring its dispersion relative to a Poisson process \cite{wang2023poisson}. The above result clearly shows that the Fano factor in this case is always strictly less than $1$. Moreover, note that the Fano factor for Poisson distribution is $1$, therefore we conclude that the customer number distribution for discrete-time ${\rm M}/{\rm M}/\infty$ is always sub-Poissonian \cite{wang2023poisson}.

To verify our results, we also use traditional method to solve discrete-time ${\rm M}/{\rm M}/\infty$ since it is already Markovian. Let $N(t)$ denote the number of customers in the system at time $t$. At each time slot, with probability $p$, one customer arrives, and each customer independently leaves the system with probability $1 - \alpha$.
Let $R(t) \sim \text{Binomial}(N(t-1), \alpha)$ be the number of customers that stay from $t-1$ to $t$, and let $A(t) \sim \text{Bernoulli}(p)$ be the number of new arrivals. Then we clearly have
\[
N(t) = R(t) + A(t).
\]
Let $G(z,t)$ be the generating function of $N(t)$ defined by
\begin{equation*}
G(z,t)=\sum_{m=0}^\infty \mathbb{P}\left(N(t)=m\right)z^m.
\end{equation*}
By the mutual independence between $R(t)$ and $A(t)$, we obtain that
\begin{equation*}
G(z,t) = \mathbb{E}[z^{R(t) + A(t)}] = \mathbb{E}[z^{R(t)}]\cdot \mathbb{E}[z^{A(t)}].
\end{equation*}
It is easy to prove that
\begin{equation*}
\begin{aligned}
\mathbb{E}[z^{R(t)}] &= G(1 - \alpha + \alpha z,t-1). \\
\mathbb{E}[z^{A(t)}] &= 1 - p + pz.
\end{aligned}
\end{equation*}
Combining the above results gives
\begin{equation}\label{i}
G(z,t)= G(1 - \alpha + \alpha z,t-1) \cdot \left(1 - p + pz\right).
\end{equation}
Since $N(0) = 0$ yields $G(z,0) = 1$, applying Eq. \eqref{i} iteratively finally yields
\[
G(z,t)=\prod_{i=1}^{t}\left[1+p\alpha^{i}(z-1)\right],\\
\]
which is fully consistent with Eq. \eqref{mm}.

\;\\

\begin{appendices}
	\section{Details of simulation}\label{ssa}
	The details of the stochastic simulation algorithm (SSA) for simulating discrete non-Markovian queueing models with general service time distribution are described as follows:
\begin{enumerate}
	\item[Step 1] Use the classical SSA to generate the stochastic trajectories of the customer up to time $t$ according to the Markovian arrival dynamics.
	\item[Step 2] Determine the customer arrival time points $\tau_1,\tau_2,\cdots,\tau_N$ before time $t$. These time points will be referred to as the arrival time points.
	\item[Step 3] Generate $N$ service times that are drawn from the service time distribution, denoted by $Y_1,Y_2,\cdots,Y_N$, and then add them to the $N$ arrival time points to obtain their service completion times, i.e. $\tau_1+Y_1,\tau_2+Y_2,\cdots,\tau_N+Y_N$.
	\item[Step 4] Determine the number of customers whose service completion time points are strictly larger than time $t$, denoted by $N_d$. Then the number of customers present at time $t$ is $N_d$.
	\item[Step 5] Use the simulated data for a large number of trajectories to obtain the numerical customer number distribution at time $t$.
\end{enumerate}
To the convenience of readers, we reemphasize here that we have assumed that customers which complete their service at time $t$ will not be counted in the customer number distribution at time $t$.
\end{appendices}

\section*{Acknowledgments}

This work was supported by grants No. 12401629 from Natural
Science Foundation of P. R. China.

\section*{Data availability}

No data was used for the research described in the article.


\bibliography{sn-bibliography}
\bibliographystyle{elsarticle-num}% common bib file

\end{document}